\documentclass[11pt]{article}
\usepackage{graphicx}
\usepackage{url}
\usepackage{amsfonts,amsthm,amsmath}
\usepackage{amssymb}
\usepackage{authblk}
\usepackage[top=1.25in, bottom=1.5in, left=1.5in, right=1.5in]{geometry}
\usepackage{tikz}
\usepackage{xcolor}

\newtheorem{Theorem}{Theorem}[section]

\newtheorem{Example}{Example}[section]

\newtheorem{Lemma}[Theorem]{Lemma}
\newtheorem{Corollary}[Theorem]{Corollary}	
	
\newtheorem{Construction}[Theorem]{Construction}

\newcommand{\zed}{{\ensuremath{\mathbb{Z}}}}

\newcommand{\QQ}{{\ensuremath{\mathcal{Q}}}} 
\newcommand{\vv}{{\ensuremath{\mathbf{v}}}} 
\newcommand{\ww}{{\ensuremath{\mathbf{w}}}}

\newcommand{\DL}{{\ensuremath{\mathsf{DL}}}}

\newcommand{\bx}{\mathbf{x}}
\newcommand{\by}{\mathbf{y}}

\newcommand{\tbl}[1]{\textcolor{blue}{#1}}
\newcommand{\trd}[1]{\textcolor{red}{#1}}

\title{Dispersed graph labellings}
\author[1]{William J.\ Martin\thanks{W.J.\ Martin’s research is supported by NSF DMS Award \#1808376.}}
\author[2]{Douglas R.\ Stinson\thanks{D.R.\ Stinson's research is supported by  NSERC discovery grant RGPIN-03882.}}
\affil[1]{Department of Mathematical Sciences\\Worcester Polytechnic Institute\\Worcester MA, 01609\\USA}
\affil[2]{David R.\ Cheriton School of Computer Science\\University of Waterloo\\ Waterloo ON, N2L 3G1\\Canada}

\date{\today}

\begin{document}

\maketitle

\begin{abstract}
A {$k$-dispersed labelling} of a graph $G$ on $n$ vertices is  a labelling of the vertices of $G$ by the integers $1, \dots , n$ such that $d(i,i+1) \geq k$ for $1 \leq i \leq n-1$. $\mathsf{DL}(G)$ denotes the maximum value of $k$ such that $G$ has a $k$-dispersed labelling. In this paper, we study upper and lower bounds on $\mathsf{DL}(G)$. Computing $\mathsf{DL}(G)$ is \textsf{NP}-hard. However, we determine the exact value 
of $\mathsf{DL}(G)$ for cycles, paths, grids, hypercubes and complete binary trees. We also give a product construction and we prove a degree-based bound.
\end{abstract}

\section{Introduction}

Many graph labelling problems have been studied over the years, starting with the graceful labellings introduced by Rosa. 
Gallian's  dynamic survey \cite{Gal} is an excellent starting point for this area of research. We assume standard graph-theoretic terminology throughout this paper, e.g., as defined in \cite{BM}.

Let $G$ be a graph having vertex set $V$, where $|V| = n$. 
Let $d(u,v)$ denote the distance between any two vertices $u$ and $v$ in $G$.
It is easy to observe that $G$ has a hamiltonian path if and only if there is a labelling of the vertices with the integers $1, \dots , n$ such that $d(i,i+1) = 1$ for $1 \leq i \leq n-1$.  Here we consider a labelling problem motivated
by the requirement that consecutively labelled vertices should be far apart.
Thus we define a \emph{$k$-dispersed labelling} to be a labelling of the vertices of $G$ by the integers $1, \dots , n$ such that $d(i,i+1) \geq k$ for $1 \leq i \leq n-1$. Equivalently, for a graph $G = (V,E)$ with $|V| = n$, we could define a $k$-dispersed labelling to be a bijection
$\phi : \{1, \dots , n\} \rightarrow V$ such that $d(\phi(i),\phi(i+1)) \geq k$ for $1 \leq i \leq n-1$.

Although it is not a main topic of this paper, we could also consider a ``circular'' variant of the above definition. 
We define a \emph{$k$-circular-dispersed labelling} to be a $k$-dispersed labelling that satisfies the additional property that $d(n,1) \geq k$.
Equivalently, a  $k$-circular-dispersed labelling could be defined to be a bijection
$\phi : \zed_n \rightarrow V$ such that $d(\phi(i),\phi(i+1)) \geq k$ for $0 \leq i \leq n-1$ 
(in this definition, for convenience, the vertices are labelled with the elements of
$\zed_n$).

As an example, we present a $4$-dispersed labelling of the $2 \times 7$ grid graph in Figure \ref{fig2,7}
(this graph is denoted as $L_{2,7}$).
This labelling is also a $4$-circular-dispersed labelling, because $d(14,1) = 4$.

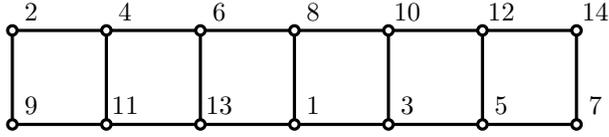
\begin{figure}
\begin{center}


\begin{tikzpicture}[scale=0.25]

\draw [very thick] (0,0) -- (5,0) -- (10,0) -- (15,0) -- (20,0) -- (25,0) -- (30,0);
\draw [very thick] (0,5) -- (5,5) -- (10,5) -- (15,5) -- (20,5) -- (25,5) -- (30,5);

\draw [very thick] (0,0) -- (0,5);
\draw [very thick] (5,0) -- (5,5);
\draw [very thick] (10,0) -- (10,5);
\draw [very thick] (15,0) -- (15,5);
\draw [very thick] (20,0) -- (20,5);
\draw [very thick] (25,0) -- (25,5);
\draw [very thick] (30,0) -- (30,5);

\draw [very thick,fill = white] (0,0) circle [radius=.25];
\draw [very thick,fill = white] (0,5) circle [radius=.25];
\draw [very thick,fill = white] (10,0) circle [radius=.25];
\draw [very thick,fill = white] (10,5) circle [radius=.25];
\draw [very thick,fill = white] (5,0) circle [radius=.25];
\draw [very thick,fill = white] (5,5) circle [radius=.25];
\draw [very thick,fill = white] (15,0) circle [radius=.25];
\draw [very thick,fill = white] (15,5) circle [radius=.25];
\draw [very thick,fill = white] (20,0) circle [radius=.25];
\draw [very thick,fill = white] (20,5) circle [radius=.25];
\draw [very thick,fill = white] (25,0) circle [radius=.25];
\draw [very thick,fill = white] (25,5) circle [radius=.25];
\draw [very thick,fill = white] (30,0) circle [radius=.25];
\draw [very thick,fill = white] (30,5) circle [radius=.25];

\node at (1,6) {\small{$2$}};
\node at (6,6) {\small{$4$}};
\node at (11,6) {\small{$6$}};
\node at (16,6) {\small{$8$}};
\node at (21,6) {\small{$10$}};
\node at (26,6) {\small{$12$}};
\node at (31,6) {\small{$14$}};

\node at (1,1) {\small{$9$}};
\node at (6,1) {\small{$11$}};
\node at (11,1) {\small{$13$}};
\node at (16,1) {\small{$1$}};
\node at (21,1) {\small{$3$}};
\node at (26,1) {\small{$5$}};
\node at (31,1) {\small{$7$}};

\end{tikzpicture}

\vspace{-.15in}

\end{center}
\caption{A $4$-dispersed labelling of $L_{2,7}$ }
\label{fig2,7}
\end{figure}

We note that most of the labelling problems discussed in \cite{Gal} do not involve distances between the vertices in a graph. One exception is the problem of \emph{radio labellings} \cite[\S 7.4]{Gal}.

Let $\mathsf{DL}(G)$ denote the maximum value of $k$ such that $G$ has a $k$-dispersed labelling and let $\mathsf{DL}^{\circ}(G)$ 
denote the maximum value of $k$ such that $G$ has a $k$-circular-dispersed labelling.
If $G$ is a finite, connected graph, then $\mathsf{DL}(G)$ and $\mathsf{DL}^{\circ}(G)$ are both well-defined positive integers. If $G$ is not connected, then it is possible that $\mathsf{DL}^{\circ}(G) = \infty$.

For the rest of the paper, $G$ always refers to a finite, connected graph. Here are two easy preliminary lemmas that we state without proof.

\begin{Lemma}
If $H$ is a spanning subgraph of $G$ (obtained by removing edges  but no vertices), then $\DL(H) \ge \DL(G)$ and  $\DL^\circ(H) \ge \DL^\circ(G)$.
\end{Lemma}
 
\begin{Lemma}
\label{circ.lem}
$\DL^\circ(G) \le \DL(G)$.
\end{Lemma}


Our first real result relates the values $\mathsf{DL}(G)$ to the distance $k$ graph of $G$. 
The \emph{distance $k$ graph} of $G$, denoted $G_k$, is the graph in which two vertices $x$ and $y$ are joined by an edge if $d(x,y) = k$. 
Clearly $G_1 = G$. Denote $G^*_{k-1} = G_1 \cup \dots\cup G_{k-1}$ and let $H_k(G) = (G^*_{k-1})^c$, the complement of graph $G^*_{k-1}$. So two vertices of $H_k(G)$ are adjacent if the distance between them (in $G$) is at least $k$.

\begin{Theorem}
\label{Thm1}
For a graph $G$, $\mathsf{DL}(G) \geq k$ if and only if $H_k(G)$ contains a hamiltonian path. Further, $\mathsf{DL}^{\circ}(G) \geq k$ if and only if $H_k(G)$ contains a hamiltonian cycle.
\end{Theorem}

\begin{proof}
Suppose that $G$ is a graph on $n$ vertices. Suppose $G$ has a $k$-dispersed labelling, say $\alpha$,  and 
define $P = (1, \dots , n)$.
For any $i$, $1 \leq i \leq n-1$, it holds that $d(i,i+1) \geq k$ (because $\alpha$ is a $k$-dispersed labelling) and hence  $\{i,i+1\}$ is an edge of $H_k(G)$. It follows that $P$ is a hamiltonian path in $H_k(G)$. 

The proof of the converse result is similar, as is the proof of the corresponding result for $\mathsf{DL}^{\circ}(G)$.
\end{proof}

\begin{Corollary}
Computing $\mathsf{DL}(G)$ is $\mathsf{NP}$-hard.
\end{Corollary}

\begin{proof}
Suppose that $\mathcal{O}$ is an oracle that computes $\mathsf{DL}(G)$ in polynomial time. We can use $\mathcal{O}$ to solve the $\mathsf{NP}$-complete  {\sc hamiltonian path} problem as follows. Given a graph $G$, run $\mathcal{O}$ on $G^c$. Observe that $H_2(G^c) = G$. 
So $\mathcal{O}(G^c) \geq 2$ if and only if 
$G$ has a hamiltonian path.
\end{proof}

The rest of the paper is organized as follows.
In Section \ref{sec2}, we determine the exact value of $\mathsf{DL}(G)$ for cycles, paths, grids, hypercubes and complete binary trees. For these classes of graphs, $\mathsf{DL}(G)=r(G)$ or $r(G) - 1$, where $r(G)$ is the radius of the graph $G$.
In Section \ref{sec3}, we give a product construction and, using this, we show that  $\DL^\circ( G\, \Box \, H) \ge \DL^\circ(G)+\DL^\circ(H)$, provided that the numbers of vertices in $G$ and $H$ is relatively prime.\footnote{The \emph{Cartesian product}, $G \, \Box \, H$, of graphs $G$ and $H$ is 
defined in Section \ref{sec3}.} In Section \ref{sec4}, we prove a degree-based lower bound on $\mathsf{DL}(G)$. Finally, in Section \ref{sec5}, we list some additional questions.

\section{Computing DL$(G)$ for some classes of graphs}
\label{sec2}

We first consider cycles, for which the values $\mathsf{DL}(G)$ can easily be determined. 

\begin{Theorem}
\label{cycle.thm}
Let $C_n$ denote a cycle of length $n$. Then  $\mathsf{DL}(C_n) = (n-1)/2$ if $n$ is odd, and 
$\mathsf{DL}(C_n) = (n-2)/2$ if $n$ is even.
\end{Theorem}

\begin{proof}
The maximum distance between two vertices of $C_n$ is $n/2$ if $n$ is even and $(n-1)/2$ if $n$ is odd. 

First, suppose that $n$ is odd and let $k = (n-1)/2.$ The graph $H_k=H_k(C_n)$ is a single (hamiltonian) cycle of length $n$, so $\mathsf{DL}(C_n) \geq k$ follows from Theorem \ref{Thm1}.
Also, $H_{k+1}(C_n)$ is the empty graph, so $\mathsf{DL}(C_n) \leq k$.

Next, suppose that $n$ is even and let $k = n/2.$ $H_k$ consists of $k$ disjoint edges, so $H_k$ is not hamiltonian and therefore $\mathsf{DL}(C_n) \leq k-1$. We now study the structure of the  graph $H_{k-1}=H_{k-1}(C_n)$, which is a cubic graph. We consider two subcases.

First, suppose $n \equiv 0 \bmod 4$. Here, the edges in $H_{k-1}$ that are not in $H_k$ form a hamiltonian cycle, so we are done. 
If $n \equiv 2 \bmod 4$, then $H_{k-1}$  is a prism; the edges in $H_{k-1}$ that are not in $H_k$ form  two disjoint cycles of length $n/2$.
It is an easy exercise to verify that the prism contains a hamiltonian path.
Thus $\mathsf{DL}(C_n) \geq k-1$ when $n$ is even, and the proof is complete.
\end{proof}

The \emph{eccentricity} of a vertex $v$ is the quantity
$\epsilon(v) = \max \{ d ( v,u ) : u \in V \}$.
  The \emph{radius} of a graph $G$, denoted $r(G)$, is the minimum eccentricity of any vertex, i.e., 
$r(G) = \min \{ \epsilon(v) : v \in V\}$.

We will require some additional related definitions for later use.
A vertex $v \in V$ is \emph{uniquely eccentric} if there is a unique vertex $u$ such that $d(u,v) = \epsilon(v)$.
A vertex $v \in V$ is a \emph{central vertex} if $\epsilon (v) = r(G)$.

\begin{Theorem}
\label{Thm4}
For a graph $G$, $\mathsf{DL}(G) \leq r(G)$.
\end{Theorem}

\begin{proof}
Suppose that $G$ has a $k$-dispersed labelling. Let $u$ be a vertex such that $\epsilon(u) = r(G)$. If $u$ has label $i < n$, then let $v$ be the vertex that is labelled $i+1$; if $u$ is labelled with $n$, then let $v$ be the vertex that is labelled $i-1$. We must have $d(u,v) \geq k$ since the labelling is $k$-dispersed. However, $d(u,v) \leq \epsilon(u) = r(G)$. Taking $k=\mathsf{DL}(G)$, 
it follows that $\mathsf{DL}(G) \leq r(G)$.
\end{proof}

We computed $\mathsf{DL}(C_n)$ for all cycles $C_n$ in Theorem \ref{cycle.thm}. It is easy to verify that 
$r(C_n) = n/2$ if $n$ is even and $r(C_n) = (n-1)/2$ if $n$ is odd. Hence, 
$\mathsf{DL}(C_n) = r(C_n)$ if $n$ is odd and $\mathsf{DL}(C_n) = r(C_n) - 1$ if $n$ is even.

\medskip


\subsection{Paths}

A simple class of graphs to consider are the paths. Let $P_m$ denote the path having $m$ edges and $m+1$ vertices.
It is easy to check that the radius of a path is given by the following formula:
\[r(P_m) = \begin{cases}
\frac{m}{2} & \text{if $m$ is even}\\
\frac{m+1}{2} & \text{if $m$ is odd.}
\end{cases}
\]

\begin{Theorem}
\label{paths.thm}
$\mathsf{DL}(P_m) = r(P_m)$ for any path $P_m$.
\end{Theorem}

\begin{proof} The vertices of the path $P_m$ will be labelled with the integers $1, \dots , m+1$.
First, suppose $m$ is even. An $\frac{m}{2}$-dispersed labelling of $P_m$ is as follows:
\[ 2 \:\:\: 4 \:\:\: \cdots \:\:\: m \:\:\: 1 \:\:\: 3 \:\:\: \cdots \:\:\: m+1.\]
For odd $m$, an $\frac{m+1}{2}$-dispersed labelling is as follows:
\[ 2 \:\:\: 4 \:\:\: \cdots \:\:\: m+1 \:\:\: 1 \:\:\: 3 \:\:\: \cdots \:\:\: m.\]

\end{proof}

\subsection{Grids}

As another, more complicated class of graphs, we consider the $m \times n$ grid graphs (or lattice graphs), which we denote by $L_{m,n}$. The graph $L_{4,6}$ is depicted in Figure \ref{fig1}.
The following lemma will be useful.

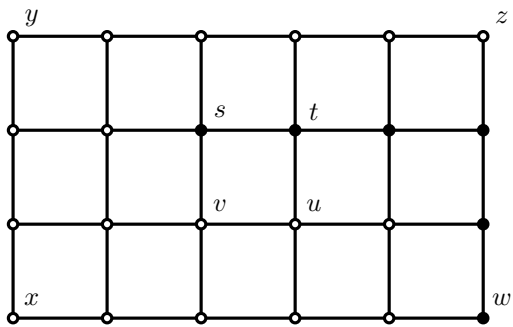
\begin{figure}
\begin{center}


\begin{tikzpicture}[scale=0.25]

\draw [very thick] (0,0) -- (5,0) -- (10,0) -- (15,0) -- (20,0) -- (25,0);
\draw [very thick] (0,5) -- (5,5) -- (10,5) -- (15,5) -- (20,5) -- (25,5);
\draw [very thick] (0,10) -- (5,10) -- (10,10) -- (15,10) -- (20,10) -- (25,10);
\draw [very thick] (0,15) -- (5,15) -- (10,15) -- (15,15) -- (20,15) -- (25,15);

\draw [very thick] (0,0) -- (0,5) -- (0,10) -- (0,15);
\draw [very thick] (5,0) -- (5,5) -- (5,10) -- (5,15);
\draw [very thick] (10,0) -- (10,5) -- (10,10) -- (10,15);
\draw [very thick] (15,0) -- (15,5) -- (15,10) -- (15,15);
\draw [very thick] (20,0) -- (20,5) -- (20,10) -- (20,15);
\draw [very thick] (25,0) -- (25,5) -- (25,10) -- (25,15);

\draw [very thick,fill = white] (0,0) circle [radius=.25];
\draw [very thick,fill = white] (0,5) circle [radius=.25];
\draw [very thick,fill = white] (0,15) circle [radius=.25];
\draw [very thick,fill = white] (0,10) circle [radius=.25];
\draw [very thick,fill = white] (10,0) circle [radius=.25];
\draw [very thick,fill = black] (10,10) circle [radius=.25];
\draw [very thick,fill = white] (10,5) circle [radius=.25];
\draw [very thick,fill = white] (10,15) circle [radius=.25];
\draw [very thick,fill = white] (5,0) circle [radius=.25];
\draw [very thick,fill = white] (5,5) circle [radius=.25];
\draw [very thick,fill = white] (5,15) circle [radius=.25];
\draw [very thick,fill = white] (5,10) circle [radius=.25];
\draw [very thick,fill = white] (15,0) circle [radius=.25];
\draw [very thick,fill = black] (15,10) circle [radius=.25];
\draw [very thick,fill = white] (15,5) circle [radius=.25];
\draw [very thick,fill = white] (15,15) circle [radius=.25];
\draw [very thick,fill = white] (20,0) circle [radius=.25];
\draw [very thick,fill = white] (20,5) circle [radius=.25];
\draw [very thick,fill = white] (20,15) circle [radius=.25];
\draw [very thick,fill = black] (20,10) circle [radius=.25];
\draw [very thick,fill = black] (25,0) circle [radius=.25];
\draw [very thick,fill = black] (25,10) circle [radius=.25];
\draw [very thick,fill = black] (25,5) circle [radius=.25];
\draw [very thick,fill = white] (25,15) circle [radius=.25];

\node at (1,1) {\small{$x$}};
\node at (26,1) {\small{$w$}};

\node at (11,6) {\small{$v$}};
\node at (16,6) {\small{$u$}};

\node at (11,11) {\small{$s$}};
\node at (16,11) {\small{$t$}};

\node at (1,16) {\small{$y$}};
\node at (26,16) {\small{$z$}};

\end{tikzpicture}

\vspace{-.15in}

\end{center}
\caption{The graph $L_{4,6}$}
\label{fig1}
\end{figure}

\begin{Lemma}
\label{three.lem}
\label{uc.lem} Suppose a graph $G$ contains three central vertices, each of which is uniquely eccentric. Then
$\mathsf{DL}(G) \leq r(G)-1$.
\end{Lemma}

\begin{proof}
Suppose that $G$ contains $n$  vertices that are labelled $1, \dots , n$. Further, suppose that $\mathsf{DL}(G) = r(G)$. At least one of the three hypothesized central vertices must receive a label $i$, where $2 \leq i \leq n-1$. Consider the vertices labelled $i-1$ and $i+1$. Since the vertex labelled $i$ is uniquely eccentric, either $d(i-1,i) < r(G)$ or $d(i,i+1) < r(G)$. This contradicts the assumption that $\mathsf{DL}(G) = r(G)$.
\end{proof}

\begin{Example}
{\rm The  graph $L_{4,6}$, has radius $r(L_{4,6}) = 5$ and there are four central vertices, namely the vertices $s,t,u$ and $v$ identified in Figure \ref{fig1}. Each of these four central vertices is uniquely eccentric: $d(s,w) = 5$, $d(u,x) = 5$,  $d(v,y) = 5$  and $d(t,z) = 5$.
(To illustrate, vertex $w$ is the only vertex that is distance five from vertex $s$. A path of length five from $s$ to $w$ is indicated by the blackened vertices in Figure \ref{fig1}.) 
Therefore, from Lemma \ref{uc.lem}, it follows that $\mathsf{DL}(L_{4,6}) \leq 4$.
We will prove a bit later that  $\mathsf{DL}(L_{4,6}) = 4$.
}\end{Example}

\medskip

More generally, we have the following upper bound.

\begin{Theorem}
\label{uc.thm}
Suppose $m$ and $n$ are both even. Then $\mathsf{DL}(L_{m,n}) \leq (m+n)/2 - 1$.
\end{Theorem}

\begin{proof}
When $m$ and $n$ are even, it is easy to see that $r(L_{m,n}) = (m+n)/2$ and this graph has four central vertices, each of which is uniquely eccentric.
Apply Lemma \ref{uc.lem}.
\end{proof}

For other values of $m$ and $n$, we have the following simple results, which we state without proof.

\bigskip

\begin{Lemma} \mbox{}
\label{gridradius.lem}
\begin{enumerate}
\item Suppose $m$ and $n$ are both odd. Then $r(L_{m,n}) = (m+n)/2 - 1$ and $L_{m,n}$ has one central vertex, which is not uniquely eccentric.
\item Suppose $m+n$ is odd. Then $r(L_{m,n}) = (m+n-1)/2$ and $L_{m,n}$ has two central vertices, neither of which is uniquely eccentric.
\end{enumerate}
\end{Lemma}
Thus Lemma \ref{uc.lem} cannot be applied in these cases, so we cannot rule out the possibility that $\mathsf{DL}(L_{m,n}) = r(L_{m,n})$ if at least one of $m$ and $n$ is odd. In fact, we will prove in this section that 
$\mathsf{DL}(L_{m,n}) = r(L_{m,n})$ if at least one of $m$ and $n$ is odd; and $\mathsf{DL}(L_{m,n}) = r(L_{m,n})-1$ if $m$ and $n$ are both even. 

First, we solve the case of $2 \times n$ grids.

\begin{Theorem}
\label{2nodd.thm}
$\mathsf{DL}(L_{2,n}) = r(L_{2,n}) = (n+1)/2$ for all odd $n \geq 3.$
\end{Theorem}

\begin{proof}
The first row of $L_{2,n}$ is labelled $2, 4, \dots , 2n$ and the second row is labelled
$n+2, n+4, \dots,2n-1, 1, 3, \dots , n$.
\end{proof}

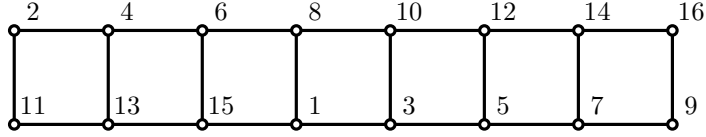
\begin{figure}
\begin{center}


\begin{tikzpicture}[scale=0.25]

\draw [very thick] (0,0) -- (5,0) -- (10,0) -- (15,0) -- (20,0) -- (25,0) -- (30,0) -- (35,0);
\draw [very thick] (0,5) -- (5,5) -- (10,5) -- (15,5) -- (20,5) -- (25,5) -- (30,5) -- (35,5);

\draw [very thick] (0,0) -- (0,5);
\draw [very thick] (5,0) -- (5,5);
\draw [very thick] (10,0) -- (10,5);
\draw [very thick] (15,0) -- (15,5);
\draw [very thick] (20,0) -- (20,5);
\draw [very thick] (25,0) -- (25,5);
\draw [very thick] (30,0) -- (30,5);
\draw [very thick] (35,0) -- (35,5);

\draw [very thick,fill = white] (0,0) circle [radius=.25];
\draw [very thick,fill = white] (0,5) circle [radius=.25];
\draw [very thick,fill = white] (10,0) circle [radius=.25];
\draw [very thick,fill = white] (10,5) circle [radius=.25];
\draw [very thick,fill = white] (5,0) circle [radius=.25];
\draw [very thick,fill = white] (5,5) circle [radius=.25];
\draw [very thick,fill = white] (15,0) circle [radius=.25];
\draw [very thick,fill = white] (15,5) circle [radius=.25];
\draw [very thick,fill = white] (20,0) circle [radius=.25];
\draw [very thick,fill = white] (20,5) circle [radius=.25];
\draw [very thick,fill = white] (25,0) circle [radius=.25];
\draw [very thick,fill = white] (25,5) circle [radius=.25];
\draw [very thick,fill = white] (30,0) circle [radius=.25];
\draw [very thick,fill = white] (30,5) circle [radius=.25];
\draw [very thick,fill = white] (35,0) circle [radius=.25];
\draw [very thick,fill = white] (35,5) circle [radius=.25];

\node at (1,6) {\small{$2$}};
\node at (6,6) {\small{$4$}};
\node at (11,6) {\small{$6$}};
\node at (16,6) {\small{$8$}};
\node at (21,6) {\small{$10$}};
\node at (26,6) {\small{$12$}};
\node at (31,6) {\small{$14$}};
\node at (36,6) {\small{$16$}};

\node at (1,1) {\small{$11$}};
\node at (6,1) {\small{$13$}};
\node at (11,1) {\small{$15$}};
\node at (16,1) {\small{$1$}};
\node at (21,1) {\small{$3$}};
\node at (26,1) {\small{$5$}};
\node at (31,1) {\small{$7$}};
\node at (36,1) {\small{$9$}};

\end{tikzpicture}

\vspace{-.15in}

\end{center}
\caption{A $4$-dispersed labelling of $L_{2,8}$ }
\label{fig2,8}
\end{figure}

\begin{Theorem}
\label{2neven.thm}
$\mathsf{DL}(L_{2,n}) = r(L_{2,n}) - 1 = n/2$ for all even $n \geq 4.$
\end{Theorem}

\begin{proof}
The first row of $L_{2,n}$ is labelled $2, 4, \dots , 2n$ and the second row is labelled
$n+3, n+5, \dots,2n-1, 1, 3, \dots , n+1$.
\end{proof}

See Figures \ref{fig2,7} and \ref{fig2,8} for illustrations of the constructions in Theorems \ref{2nodd.thm} and \ref{2neven.thm}, respectively.

\medskip

Suppose $m\geq 4$ is even. We can construct optimal labellings of $m \times n$ grids from optimal labellings of $2  \times n$ grids recursively. 

\begin{Theorem}
\label{evenm.thm}
Suppose $m\geq 2$ is even.
Then \[\mathsf{DL}(L_{m,n}) = 
\begin{cases}
\frac{m+n-2}{2} & \text{if $n$ is even} \\
\frac{m+n-1}{2} & \text{if $n$ is odd.} 
\end{cases}
\]
\end{Theorem}
\begin{proof}
For odd $n$, we start with the $\frac{n+1}{2}$-dispersed labelling of $L_{2,n}$ constructed in Theorem \ref{2nodd.thm};
for even $n$, we start with the $\frac{n}{2}$-dispersed labelling of $L_{2,n}$ constructed in Theorem \ref{2neven.thm}.
Let the first row of one of these optimal labellings be denoted $A$ and let the second row be denoted $B$. For an integer $i$, let $A+i$ ($B+i$, resp.) denote $A$ ($B$, resp) with $i$ added to every label. 
Then construct the labelling of $L_{m,n}$ having the  $m$ rows indicated in Figure \ref{interleave1.fig}. 

See Figure \ref{fig4,7} for an illustration of the construction when $m=4$ and $n = 7$.

 Basically, we have interleaved $m/2$ isomorphic copies of the 
 labelling of $L_{2,n}$. For $1 \leq k \leq \frac{m}{2}$, 
 rows $k$ and $\frac{m}{2} + k$ comprise a copy of of the labelling of $L_{2,n}$ that use the labels
 $\{ 2(k-1)n + 1 , \dots , 2kn\}$. This labelling adds $2(k-1)n$ to each label in the ``original'' labelling of $L_{2,n}$. Note also that the two rows in each copy of $L_{2,n}$ are separated by $\frac{m}{2} -1$ other rows.
 
 It is not hard to prove that the result is an $\frac{m+n-2}{2}$-dispersed labelling of $L_{m,n}$ when $n$ is even, and an
 $\frac{m+n-1}{2}$-dispersed labelling of $L_{m,n}$ when $n$ is odd. 
 The interleaving increases the distances between consecutively labelled vertices within a particular 
copy of the $L_{2,n}$ by $m/2 - 1$. (This is because, for any two consecutively labelled vertices, one is in an $A+2kn$ (for some $k$) and the other is in $B+2kn$, so the vertical distance between them has increased from one 
(in the original $L_{2,n}$) to $m/2$.) In the case of even $n$, the minimum distance of $n/2$ is increased to  $n/2 + m/2 -1 = (m+n-2)/2$, as desired. For odd $n$, 
 the minimum distance of $(n+1)/2$ is increased to  $(n+1)/2 + m/2 -1 = (m+n-1)/2$, as desired.
 
 It is also necessary to consider the distance between the ``last vertex'' (i.e., the vertex with the largest label) in one copy of $L_{2,n}$ and the ``first'' vertex (i.e., the vertex with the smallest label) in the next copy. 
 The first vertex in a copy of of $L_{2,n}$ is  the middle element (when $n$ is odd) or the leftmost of the two middle elements (when $n$ is even) of row containing $B+2kn$, i.e., row $\frac{m}{2}+k$. The last  vertex is the rightmost vertex in the row containing $A + 2(k-1)n$, i.e., row $k-1$. The distance between these two vertices in the labelling of $L_{2,n}$ is $(n+1)/2$ when $n$ is odd, and $(n+2)/2$ when $n$ is even. The interleaving changes the vertical distance between the two vertices from $1$ to ${m}/{2}+k - (k-1) = {m}/{2} +1$.  The result is that the distance increases by $m/2$, which is more than what is required in the resulting labelling of
 $L_{m,n}$.
\end{proof}

 \begin{figure}

\[
\begin{array}{c|c|}
\cline{2-2} 
\text{row 1} & A \\ \cline{2-2}
\text{row 2} & A+2n \\ \cline{2-2}  
&  \vdots \\ \cline{2-2}  
\text{row $\frac{m}{2}$} & A + (m-2)n\\ \cline{2-2}  
\text{row $\frac{m}{2} + 1$} & B\\ \cline{2-2}  
\text{row $\frac{m}{2} + 2$} & B+2n\\ \cline{2-2}  
 &\vdots\\ \cline{2-2}  
\text{row $m$} & B + (m-2)n \\ \cline{2-2}
\end{array}
\]


\caption{Interleaved labellings of $L_{2,n}$}
\label{interleave1.fig}
\end{figure}

We now study $L_{m,n}$ for odd $m$. We can also assume that $n$ is odd, because the case of odd $m$ and even $n$ is 
equivalent to the case where $m$ is even and $n$ is odd, and this is covered by Theorem \ref{evenm.thm}.
We begin with $m=3$.


\begin{Theorem}
\label{m=3.thm}
$\mathsf{DL}(L_{3,n}) = r(L_{3,n}) = (n+1)/2$ for all odd $n \geq 3$.
\end{Theorem}

\begin{proof}
It suffices to present an $\frac{n+1}{2}$-dispersed labelling of $L_{3,n}$ when $n$ is odd.

First, suppose that $n \not\equiv 0 \bmod 3$. To avoid cumbersome language, we refer to the vertex labelled $j$ as ``vertex $j$,'' etc. Vertex 1 is in the bottom left corner of the grid.
To get from vertex $j$ to vertex $j+1$, we proceed $(n-1)/2$ vertices to the right and one row up, wrapping around if necessary. 
So the distance from vertex $j$ to vertex $j+1$ is at least $(n-1)/2 + 1 = (n+1)/2$. (Observe that wrapping around a row and/or column cannot decrease the distance between the two vertices. If vertex $j+1$ is in a smaller column than vertex $j$, due to wrapping, then the horizontal distance between vertices $j$ and $j+1$ is $(n+1)/2$ instead of $(n-1)/2$. 
If vertex $j$ is in the top row, then vertex $j+1$ is in the bottom row and the vertical distance between these two vertices is two instead of one.) 

See Figure \ref{fig3,5} for the labelling obtained when $n=5$.

We need to verify that all $3n$ vertices are labelled exactly once. Suppose that vertices $j$ and $j'$ are identical. Then $j \equiv j' \bmod 3$ and $(n-1)j/2 \equiv (n-1)j'/2 \bmod n$. Since $\gcd((n-1)/2,n) = 1$, the second congruence becomes $j \equiv j' \bmod n$. Since $\gcd(n,3) = 1$, the two congruences imply that $j \equiv j' \bmod 3n$. 

It remains to consider the cases where $n \equiv 0 \bmod 3$. If we follow the procedure described above, vertices $1$ and $n+1$ will be identical (as will various other pairs of vertices). The following modification will rectify this problem.
\begin{enumerate}
\item Vertex  1 is in the bottom left corner of the grid. Label vertices $2, \dots ,n$ as before:
to get from vertex $j$ to vertex $j+1$, proceed $(n-1)/2$ vertices to the right and one row up, wrapping around if necessary.
\item For $n+1 \leq j \leq 3n$,  vertex  $j$ is immediately to the right of the vertex labelled $j-n$ (wrapping around if necessary).
\end{enumerate} 

See Figure \ref{fig3,9} for the labelling obtained when $n=9$.

Most of the verifications are the same as before.  However, as special cases, we need to check the distance between vertices $n$ and $n+1$, and the distance between vertices $2n$ and $2n+1$.
In both of these cases, they are $(n-1)/2$ columns apart and two rows apart, so the distance between them is $(n-1)/2 + 2 = (n+3)/2$.
\end{proof}

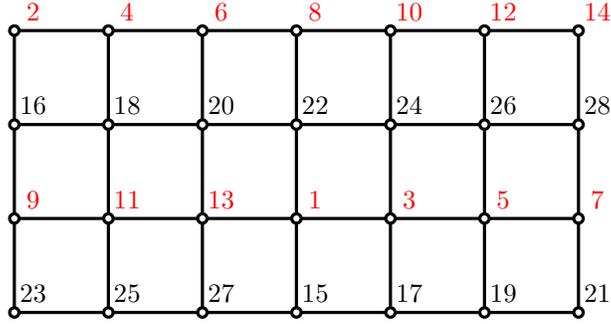
\begin{figure}
\begin{center}


\begin{tikzpicture}[scale=0.25]

\draw [very thick] (0,0) -- (5,0) -- (10,0) -- (15,0) -- (20,0) -- (25,0) -- (30,0);
\draw [very thick] (0,5) -- (5,5) -- (10,5) -- (15,5) -- (20,5) -- (25,5) -- (30,5);
\draw [very thick] (0,10) -- (5,10) -- (10,10) -- (15,10) -- (20,10) -- (25,10) -- (30,10);
\draw [very thick] (0,15) -- (5,15) -- (10,15) -- (15,15) -- (20,15) -- (25,15) -- (30,15);

\draw [very thick] (0,0) -- (0,5) -- (0,10) -- (0,15);
\draw [very thick] (5,0) -- (5,5) -- (5,10) -- (5,15);
\draw [very thick] (10,0) -- (10,5) -- (10,10) -- (10,15);
\draw [very thick] (15,0) -- (15,5) -- (15,10) -- (15,15);
\draw [very thick] (20,0) -- (20,5) -- (20,10) -- (20,15);
\draw [very thick] (25,0) -- (25,5) -- (25,10) -- (25,15);
\draw [very thick] (30,0) -- (30,5) -- (30,10) -- (30,15);

\draw [very thick,fill = white] (0,0) circle [radius=.25];
\draw [very thick,fill = white] (0,5) circle [radius=.25];
\draw [very thick,fill = white] (10,0) circle [radius=.25];
\draw [very thick,fill = white] (10,5) circle [radius=.25];
\draw [very thick,fill = white] (5,0) circle [radius=.25];
\draw [very thick,fill = white] (5,5) circle [radius=.25];
\draw [very thick,fill = white] (15,0) circle [radius=.25];
\draw [very thick,fill = white] (15,5) circle [radius=.25];
\draw [very thick,fill = white] (20,0) circle [radius=.25];
\draw [very thick,fill = white] (20,5) circle [radius=.25];
\draw [very thick,fill = white] (25,0) circle [radius=.25];
\draw [very thick,fill = white] (25,5) circle [radius=.25];
\draw [very thick,fill = white] (30,0) circle [radius=.25];
\draw [very thick,fill = white] (30,5) circle [radius=.25];

\draw [very thick,fill = white] (0,10) circle [radius=.25];
\draw [very thick,fill = white] (0,15) circle [radius=.25];
\draw [very thick,fill = white] (10,10) circle [radius=.25];
\draw [very thick,fill = white] (10,15) circle [radius=.25];
\draw [very thick,fill = white] (5,10) circle [radius=.25];
\draw [very thick,fill = white] (5,15) circle [radius=.25];
\draw [very thick,fill = white] (15,10) circle [radius=.25];
\draw [very thick,fill = white] (15,15) circle [radius=.25];
\draw [very thick,fill = white] (20,10) circle [radius=.25];
\draw [very thick,fill = white] (20,15) circle [radius=.25];
\draw [very thick,fill = white] (25,10) circle [radius=.25];
\draw [very thick,fill = white] (25,15) circle [radius=.25];
\draw [very thick,fill = white] (30,10) circle [radius=.25];
\draw [very thick,fill = white] (30,15) circle [radius=.25];

\node at (1,16) {\small{\trd{$2$}}};
\node at (6,16) {\small{\trd{$4$}}};
\node at (11,16) {\small{\trd{$6$}}};
\node at (16,16) {\small{\trd{$8$}}};
\node at (21,16) {\small{\trd{$10$}}};
\node at (26,16) {\small{\trd{$12$}}};
\node at (31,16) {\small{\trd{$14$}}};

\node at (1,6) {\small{\trd{$9$}}};
\node at (6,6) {\small{\trd{$11$}}};
\node at (11,6) {\small{\trd{$13$}}};
\node at (16,6) {\small{\trd{$1$}}};
\node at (21,6) {\small{\trd{$3$}}};
\node at (26,6) {\small{\trd{$5$}}};
\node at (31,6) {\small{\trd{$7$}}};

\node at (1,11) {\small{$16$}};
\node at (6,11) {\small{$18$}};
\node at (11,11) {\small{$20$}};
\node at (16,11) {\small{$22$}};
\node at (21,11) {\small{$24$}};
\node at (26,11) {\small{$26$}};
\node at (31,11) {\small{$28$}};

\node at (1,1) {\small{$23$}};
\node at (6,1) {\small{$25$}};
\node at (11,1) {\small{$27$}};
\node at (16,1) {\small{$15$}};
\node at (21,1) {\small{$17$}};
\node at (26,1) {\small{$19$}};
\node at (31,1) {\small{$21$}};

\end{tikzpicture}

\vspace{-.15in}

\end{center}
\caption{A $4$-dispersed labelling of $L_{4,7}$ }
\label{fig4,7}
\end{figure}



\begin{figure}
\begin{center}


\begin{tikzpicture}[scale=0.25]

\draw [very thick] (0,0) -- (5,0) -- (10,0) -- (15,0) -- (20,0) ;
\draw [very thick] (0,5) -- (5,5) -- (10,5) -- (15,5) -- (20,5) ;
\draw [very thick] (0,10) -- (5,10) -- (10,10) -- (15,10) -- (20,10);

\draw [very thick] (0,0) -- (0,5) -- (0,10);
\draw [very thick] (5,0) -- (5,5) -- (5,10);
\draw [very thick] (10,0) -- (10,5) -- (10,10);
\draw [very thick] (15,0) -- (15,5) -- (15,10);
\draw [very thick] (20,0) -- (20,5) -- (20,10);

\draw [very thick,fill = white] (0,0) circle [radius=.25];
\draw [very thick,fill = white] (0,5) circle [radius=.25];
\draw [very thick,fill = white] (0,10) circle [radius=.25];
\draw [very thick,fill = white] (10,0) circle [radius=.25];
\draw [very thick,fill = white] (10,10) circle [radius=.25];
\draw [very thick,fill = white] (10,5) circle [radius=.25];
\draw [very thick,fill = white] (5,0) circle [radius=.25];
\draw [very thick,fill = white] (5,5) circle [radius=.25];
\draw [very thick,fill = white] (5,10) circle [radius=.25];
\draw [very thick,fill = white] (15,0) circle [radius=.25];
\draw [very thick,fill = white] (15,10) circle [radius=.25];
\draw [very thick,fill = white] (15,5) circle [radius=.25];
\draw [very thick,fill = white] (20,0) circle [radius=.25];
\draw [very thick,fill = white] (20,5) circle [radius=.25];
\draw [very thick,fill = white] (20,10) circle [radius=.25];

\node at (1,6) {\small{$11$}};
\node at (6,6) {\small{$14$}};
\node at (11,6) {\small{$2$}};
\node at (16,6) {\small{$5$}};
\node at (21,6) {\small{$8$}};

\node at (1,11) {\small{$6$}};
\node at (6,11) {\small{$9$}};
\node at (11,11) {\small{$12$}};
\node at (16,11) {\small{$15$}};
\node at (21,11) {\small{$3$}};

\node at (1,1) {\small{$1$}};
\node at (6,1) {\small{$4$}};
\node at (11,1) {\small{$7$}};
\node at (16,1) {\small{$10$}};
\node at (21,1) {\small{$13$}};

\end{tikzpicture}

\vspace{-.15in}

\end{center}
\caption{A $3$-dispersed labelling of $L_{3,5}$ }
\label{fig3,5}
\end{figure}
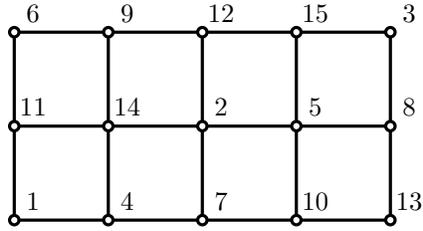

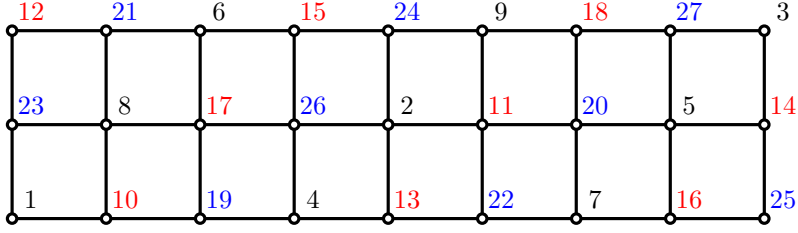
\begin{figure}
\begin{center}


\begin{tikzpicture}[scale=0.25]

\draw [very thick] (0,0) -- (5,0) -- (10,0) -- (15,0) -- (20,0) -- (25,0) -- (30,0) -- (35,0) -- (40,0);
\draw [very thick] (0,5) -- (5,5) -- (10,5) -- (15,5) -- (20,5) -- (25,5) -- (30,5) -- (35,5) -- (40,5);
\draw [very thick] (0,10) -- (5,10) -- (10,10) -- (15,10) -- (20,10) -- (25,10) -- (30,10) -- (35,10) -- (40,10);

\draw [very thick] (0,0) -- (0,5) -- (0,10);
\draw [very thick] (5,0) -- (5,5) -- (5,10);
\draw [very thick] (10,0) -- (10,5) -- (10,10);
\draw [very thick] (15,0) -- (15,5) -- (15,10);
\draw [very thick] (20,0) -- (20,5) -- (20,10);
\draw [very thick] (25,0) -- (25,5) -- (25,10);
\draw [very thick] (30,0) -- (30,5) -- (30,10);
\draw [very thick] (35,0) -- (35,5) -- (35,10);
\draw [very thick] (40,0) -- (40,5) -- (40,10);

\draw [very thick,fill = white] (0,0) circle [radius=.25];
\draw [very thick,fill = white] (0,5) circle [radius=.25];
\draw [very thick,fill = white] (0,10) circle [radius=.25];
\draw [very thick,fill = white] (10,0) circle [radius=.25];
\draw [very thick,fill = white] (10,10) circle [radius=.25];
\draw [very thick,fill = white] (10,5) circle [radius=.25];
\draw [very thick,fill = white] (5,0) circle [radius=.25];
\draw [very thick,fill = white] (5,5) circle [radius=.25];
\draw [very thick,fill = white] (5,10) circle [radius=.25];
\draw [very thick,fill = white] (15,0) circle [radius=.25];
\draw [very thick,fill = white] (15,10) circle [radius=.25];
\draw [very thick,fill = white] (15,5) circle [radius=.25];
\draw [very thick,fill = white] (20,0) circle [radius=.25];
\draw [very thick,fill = white] (20,5) circle [radius=.25];
\draw [very thick,fill = white] (20,10) circle [radius=.25];
\draw [very thick,fill = white] (25,0) circle [radius=.25];
\draw [very thick,fill = white] (25,5) circle [radius=.25];
\draw [very thick,fill = white] (25,10) circle [radius=.25];
\draw [very thick,fill = white] (35,0) circle [radius=.25];
\draw [very thick,fill = white] (35,10) circle [radius=.25];
\draw [very thick,fill = white] (35,5) circle [radius=.25];
\draw [very thick,fill = white] (30,0) circle [radius=.25];
\draw [very thick,fill = white] (30,5) circle [radius=.25];
\draw [very thick,fill = white] (30,10) circle [radius=.25];
\draw [very thick,fill = white] (40,0) circle [radius=.25];
\draw [very thick,fill = white] (40,5) circle [radius=.25];
\draw [very thick,fill = white] (40,10) circle [radius=.25];

\node at (1,11) {\small{\trd{$12$}}};
\node at (6,11) {\small{\tbl{$21$}}};
\node at (11,11) {\small{$6$}};
\node at (16,11) {\small{\trd{$15$}}};
\node at (21,11) {\small{\tbl{$24$}}};
\node at (26,11) {\small{9}};
\node at (31,11) {\small{\trd{$18$}}};
\node at (36,11) {\small{\tbl{$27$}}};
\node at (41,11) {\small{$3$}};

\node at (1,6) {\small{\tbl{$23$}}};
\node at (6,6) {\small{$8$}};
\node at (11,6) {\small{\trd{$17$}}};
\node at (16,6) {\small{\tbl{$26$}}};
\node at (21,6) {\small{$2$}};
\node at (26,6) {\small{\trd{$11$}}};
\node at (31,6) {\small{\tbl{$20$}}};
\node at (36,6) {\small{$5$}};
\node at (41,6) {\small{\trd{$14$}}};

\node at (1,1) {\small{$1$}};
\node at (6,1) {\small{\trd{$10$}}};
\node at (11,1) {\small{\tbl{$19$}}};
\node at (16,1) {\small{$4$}};
\node at (21,1) {\small{\trd{$13$}}};
\node at (26,1) {\small{\tbl{$22$}}};
\node at (31,1) {\small{$7$}};
\node at (36,1) {\small{\trd{$16$}}};
\node at (41,1) {\small{\tbl{$25$}}};

\end{tikzpicture}

\vspace{-.15in}

\end{center}
\caption{A $5$-dispersed labelling of $L_{3,9}$ }
\label{fig3,9}
\end{figure}


\begin{Theorem}
\label{oddm.thm}
$\mathsf{DL}(L_{m,n}) = r(L_{m,n}) = (n+m)/2 - 1$ for all odd $m,n \geq 3$.
\end{Theorem}

\begin{proof}
In view of Theorem \ref{m=3.thm}, we can assume that $m \geq 5$. 
Denote $m = 3 + 2t$; then $t = (m-3)/2$ and $t \geq 1$.

We will use the $(n+1)/2$-dispersed labelling of $L_{2,n}$ given in the proof of Theorem \ref{evenm.thm}. Let $A$ be the top row and let $B$ be the bottom row in this labelling.

We also use a modification of the $(n+1)/2$-dispersed labelling of $L_{3,n}$ constructed in Theorem \ref{m=3.thm}, where we by replace every label $j$ by $3n+1 - j$. The vertex in the bottom left corner originally had the label 1. Therefore, in the modified $(n+1)/2$-dispersed labelling, the vertex in the bottom left corner has the label $3n$.
Denote the rows of the \emph{modified} $(n+1)/2$-dispersed labelling by  $C$ (top row), $D$ (middle row) and $E$ (bottom row).

Then construct the labelling of $L_{m,n}$ having the  $m = 2t+3$ rows indicated in Figure \ref{interleave2.fig}.
In this labelling of $L_{m,n}$, the top row is $E$ and the bottom row is $C$ (that is, the bottom row of the modified labelling of
$L_{3,n}$ is used as the top row of the  labelling of $L_{m,n}$, etc.).

See Figure \ref{fig5,5} for an illustration of the construction when $m=5$ and $n = 5$. We give a bit more detail regarding the top row in Figure \ref{fig5,5}. The bottom row of Figure \ref{fig3,5}, namely, $1 \:\: 4 \:\: 7 \:\: 10 \:\: 13$, is relabelled as
$15 \:\: 12 \:\: 9 \:\: 6 \:\: 3$. This becomes the top row of Figure \ref{fig5,5}.

We are interleaving an $\frac{n+1}{2}$-dispersed labelling of $L_{3,n}$ with $t$ copies of an $\frac{n+1}{2}$-dispersed labelling of $L_{2,n}$.
We need to prove that the distance between consecutively labelled vertices (say $j$ and $j+1$) is at least $(m+n-2)/2$. The proof naturally divides into a number of cases:
\begin{description}
\item[case 1] Vertices $j$ and $j+1$ are both in the same copy of $L_{2,n}$.
\item[case 2] Vertex $j$ is the last (i.e., highest-numbered) vertex in one copy of $L_{2,n}$ and vertex $j+1$ is the first (i.e., lowest-numbered) vertex in the next copy of $L_{2,n}$.
\item[case 3] Vertices $j$ and $j+1$ are both in the $L_{3,n}$.
\item[case 4] Vertex $j$ is the last vertex in the $L_{3,n}$ and vertex $j+1$ is the first vertex in the first copy of $L_{2,n}$ 
(i.e., $j = 3n$).
\end{description}

Case 1 is similar to the analogous case in the proof of Theorem \ref{evenm.thm}. The two rows of one copy of a labelling of $L_{2,n}$ are separated by $t$ additional rows in the labelling of the $L_{m,n}$. It follows that the distance between vertices $j$ and $j+1$ in
$L_{m,n}$ is $t$ larger than the distance in $L_{2,n}$. So the distance is at least 
\[ \frac{n+1}{2} + t = \frac{n+1 + m-3}{2} = \frac{m+n-2}{2}.\]

\smallskip

Case 2 is also similar to the analogous case in the proof of Theorem \ref{evenm.thm}. 

\smallskip

For case 3, the distance between vertices $j$ and $j+1$ in
$L_{m,n}$ is at least $t$ larger than the distance in the labelling of the $L_{3,n}$, since the two relevant rows are separated by at least $t$ new rows.
Therefore, as in case 1, the distance is at least $(m+n-2)/2$.

\smallskip

For case 4, we use the fact that vertex $3n$ is in row $1$ (the row designated as $E$) and column $1$, and vertex $3n+1$ is in row $t+3$ (the row designated as $B+3n$) and column $(n+1)/2$.
So the distance between these two vertices is at least 
\[ \frac{n-1}{2} + t+2 = \frac{n-1 + m-3}{2} + 2 = \frac{m+n}{2}.\]
This completes the proof.
\end{proof}

\begin{figure}

\[
\begin{array}{c|c|}
\cline{2-2}
\text{row 1} & E\\ \cline{2-2}
\text{row 2} &  A+3n \\ \cline{2-2}  
\text{row 3} &  A+5n \\ \cline{2-2}
&  \vdots \\ \cline{2-2}  
\text{row $t+1$} &     A + (2t+1)n\\ \cline{2-2}  
\text{row $t+2$} &  D \\ \cline{2-2}
\text{row $t+3$}  &  B+3n\\ \cline{2-2}  
\text{row $t+4$} & B+5n \\ \cline{2-2}
&  \vdots\\ \cline{2-2}  
\text{row $2t+2$} & B + (2t+1)n \\ \cline{2-2}
\text{row $2t+3$} & C \\ \cline{2-2}
\end{array}
\]


\caption{Interleaved labellings of $L_{2,n}$ and $L_{3,n}$}
\label{interleave2.fig}
\end{figure}

\subsection{Hypercubes}

\begin{figure}
\begin{center}


\begin{tikzpicture}[scale=0.25]

\draw [very thick] (0,0) -- (5,0) -- (10,0) -- (15,0) -- (20,0) ;
\draw [very thick] (0,5) -- (5,5) -- (10,5) -- (15,5) -- (20,5) ;
\draw [very thick] (0,10) -- (5,10) -- (10,10) -- (15,10) -- (20,10);
\draw [very thick] (0,15) -- (5,15) -- (10,15) -- (15,15) -- (20,15) ;
\draw [very thick] (0,20) -- (5,20) -- (10,20) -- (15,20) -- (20,20);

\draw [very thick] (0,0) -- (0,5) -- (0,10) -- (0,15) -- (0,20);
\draw [very thick] (5,0) -- (5,5) -- (5,10) -- (5,15) -- (5,20);
\draw [very thick] (10,0) -- (10,5) -- (10,10) -- (10,15) -- (10,20);
\draw [very thick] (15,0) -- (15,5) -- (15,10) -- (15,15) -- (15,20);
\draw [very thick] (20,0) -- (20,5) -- (20,10) -- (20,15) -- (20,20);

\draw [very thick,fill = white] (0,0) circle [radius=.25];
\draw [very thick,fill = white] (0,5) circle [radius=.25];
\draw [very thick,fill = white] (0,10) circle [radius=.25];
\draw [very thick,fill = white] (0,15) circle [radius=.25];
\draw [very thick,fill = white] (0,20) circle [radius=.25];
\draw [very thick,fill = white] (10,0) circle [radius=.25];
\draw [very thick,fill = white] (10,10) circle [radius=.25];
\draw [very thick,fill = white] (10,15) circle [radius=.25];
\draw [very thick,fill = white] (10,20) circle [radius=.25];
\draw [very thick,fill = white] (10,5) circle [radius=.25];
\draw [very thick,fill = white] (5,0) circle [radius=.25];
\draw [very thick,fill = white] (5,5) circle [radius=.25];
\draw [very thick,fill = white] (5,10) circle [radius=.25];
\draw [very thick,fill = white] (5,15) circle [radius=.25];
\draw [very thick,fill = white] (5,20) circle [radius=.25];
\draw [very thick,fill = white] (15,0) circle [radius=.25];
\draw [very thick,fill = white] (15,10) circle [radius=.25];
\draw [very thick,fill = white] (15,5) circle [radius=.25];
\draw [very thick,fill = white] (15,20) circle [radius=.25];
\draw [very thick,fill = white] (15,15) circle [radius=.25];
\draw [very thick,fill = white] (20,0) circle [radius=.25];
\draw [very thick,fill = white] (20,5) circle [radius=.25];
\draw [very thick,fill = white] (20,10) circle [radius=.25];
\draw [very thick,fill = white] (20,15) circle [radius=.25];
\draw [very thick,fill = white] (20,20) circle [radius=.25];

\node at (1,11) {\small{\trd{$5$}}};
\node at (6,11) {\small{\trd{$2$}}};
\node at (11,11) {\small{\trd{$14$}}};
\node at (16,11) {\small{\trd{$11$}}};
\node at (21,11) {\small{\trd{$8$}}};

\node at (1,21) {\small{\trd{$15$}}};
\node at (6,21) {\small{\trd{$12$}}};
\node at (11,21) {\small{\trd{$9$}}};
\node at (16,21) {\small{\trd{$6$}}};
\node at (21,21) {\small{\trd{$3$}}};

\node at (1,1) {\small{\trd{$10$}}};
\node at (6,1) {\small{\trd{$7$}}};
\node at (11,1) {\small{\trd{$4$}}};
\node at (16,1) {\small{\trd{$1$}}};
\node at (21,1) {\small{\trd{$13$}}};

\node at (1,6) {\small{$22$}};
\node at (6,6) {\small{$24$}};
\node at (11,6) {\small{$16$}};
\node at (16,6) {\small{$18$}};
\node at (21,6) {\small{$20$}};

\node at (1,16) {\small{$17$}};
\node at (6,16) {\small{$19$}};
\node at (11,16) {\small{$21$}};
\node at (16,16) {\small{$23$}};
\node at (21,16) {\small{$25$}};

\end{tikzpicture}

\vspace{-.15in}

\end{center}
\caption{A $4$-dispersed labelling of $L_{5,5}$ }
\label{fig5,5}
\end{figure}
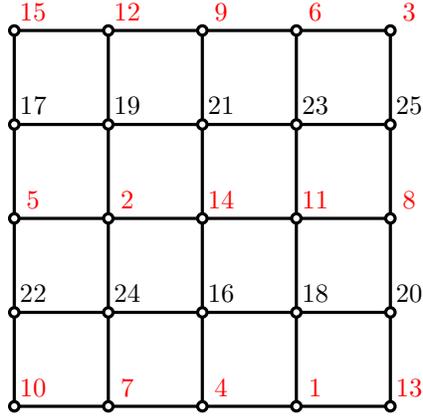

The \emph{$n$-dimensional hypercube}, denoted $\QQ_n$, is a graph having vertex set $V = (\zed_2)^n$.
Two vertices are adjacent in $\QQ_n$ if they differ in exactly one co-ordinate.
For $\vv, \ww \in V$, where
$\ww = (v_1, \dots , v_n)$ and $\ww = (w_1, \dots , w_n)$, it is easy to see that  
\[d(v,w) = | \{ i: v_i \neq w_i \} |.\] 
Here are some other easily verified properties of $\QQ_n$. 

\begin{Lemma}
\label{ncube.lem}
Let $ \vv \in (\zed_2)^n$ and suppose $0 \leq i \leq n$. Then there are precisely $\binom{n}{i}$ vertices $\ww$ such that $d(\vv,\ww) = i$ 
in the graph $\QQ_n$.
\end{Lemma}

\begin{Corollary}
\label{ncube.cor}
For any integer $n \geq 2$, $r(\QQ_n) = n$. Further, every vertex is a uniquely eccentric central vertex. 
\end{Corollary}

Lemma \ref{three.lem} and Corollary \ref{ncube.cor} immediately imply that $\DL(\QQ_n) \leq n-1$ for all $n \geq 2$.
We will prove that $\DL (\QQ_n) \geq n-1$ by constructing a suitable labelling of the vertices of $\QQ_n$.
It suffices to find a permutation of the $n$-tuples in $V$, say $\vv_1, \dots , \vv_{2^n}$, such that 
$d (\vv_j, \vv_{j+1}) \geq n-1$ for $1 \leq j \leq 2^n-1$. In the resulting dispersed labelling, vertex $\vv_j$ is labelled $j$, for $1 \leq j \leq 2^n$.

We construct the permutations recursively. In order for the recursive construction to work, we require some additional properties. Hence, for each $n \geq 2$, we will construct a particular permutation $\Pi_n = (\vv_1, \dots , \vv_{2^n})$ satisfying the following  properties.
\begin{enumerate}
\item $d (\vv_j, \vv_{j+1}) \geq n-1$ for $1 \leq j \leq 2^n-1$,
\item $\vv_1 = (0, \dots , 0)$,
\item $\vv_{2^{n-1}} = (1, \dots , 1, 0)$,
\item $\vv_{2^{n-1}+1} = (0, \dots , 0, 1)$,
\item $\vv_{2^n} = (1, \dots , 1)$.
\end{enumerate}

\begin{Construction}
\label{perm.const}
{\rm 
We describe how to construct the permutations $\Pi_n$ that satisfy the five properties enumerated above. To begin, when $n = 2$, $\Pi_2$ is the following permutation:
\[
(0, 0), \: (1,  0), \: (0,  1), \: (1, 1).
\]
Now we construct $\Pi_{n+1}$ from $\Pi_n$. For notational convenience, we denote 
$\Pi_n = (\vv_1, \dots , \vv_{2^n})$ and $\Pi_{n+1} = (\ww_1, \dots , \ww_{2^{n+1}})$.
\begin{itemize}
\item For $1 \leq j \leq 2^{n-1}$, define 
\[\ww_j = 
\begin{cases} (\vv_j, 0) & \text{ if $j$ is odd}\\
 (\vv_j, 1) & \text{if $j$ is even}.
 \end{cases}
 \]
\item For $2^{n-1}+1 \leq j \leq 2^n$, define 
\[\ww_j = 
\begin{cases} (\vv_j, 1) & \text{ if $j$ is odd}\\
 (\vv_j, 0) & \text{if $j$ is even}.
 \end{cases}
 \]
\item For $2^n + 1 \leq j \leq 2^n + 2^{n-1}$, define 
\[\ww_j = 
\begin{cases} (\vv_{j-2^n}, 1) & \text{ if $j$ is odd}\\
 (\vv_{j-2^n}, 0) & \text{if $j$ is even}.
 \end{cases}
 \]
\item For $2^n + 2^{n-1} +1 \leq j \leq 2^{n+1}$, define 
\[\ww_j = 
\begin{cases} (\vv_{j-2^n}, 0) & \text{ if $j$ is odd}\\
 (\vv_{j-2^n}, 1) & \text{if $j$ is even}.
 \end{cases}
 \]
\end{itemize}
}
\end{Construction}

\begin{Example}
{\rm $\Pi_3$ is the following permutation:
\[
(0, 0, 0), \: (1,  0, 1), \: (0,  1, 1), \: (1, 1, 0), \: (0, 0, 1), \: (1,  0, 0), \: (0,  1, 0), \: (1, 1, 1).
\]
}\end{Example}

\begin{Theorem}
\label{ncube.thm}
For all $n \geq 2$, $\DL(\QQ_n) = n-1$.
\end{Theorem}

\begin{proof}
We have already observed that $\DL(\QQ_n) \leq n-1$, so it suffices to prove that $\DL(\QQ_n) \geq n-1$. This follows by showing that each of the 
permutations $\Pi_n$ obtained from Construction \ref{perm.const} satisfies the properties 1--5 enumerated above, which we prove by induction on $n$.

First, we observe that each $\Pi_n$ is a permutation. 
This is because every $n$-tuple in $\Pi_n$ is appended with both a $0$ and a $1$ in the construction of $\Pi_{n+1}$. 

For $n=2$, the stated properties can be verified easily. Now, suppose that properties 1--5 hold for a particular $\Pi_n$, where $n \geq  2$.
We will show that properties 1--5 hold for $\Pi_{n+1}$. We use the notation from Construction \ref{perm.const}.

When $j \neq 2^n$, $\ww_j$ and $\ww_{j+1}$ are obtained from two consecutive $\vv_j$'s, which are assumed to have distance at least $n-1$, by induction.  In most cases, one of $\ww_j$ and $\ww_{j+1}$ is appended with a $1$ and the other is appended with a $0$. 
In these cases, $d(\ww_j, \ww_{j+1}) \geq n$.

The only cases where $\ww_j$ and $\ww_{j+1}$ are appended with the same symbol are when $j = 2^{n-1}, 2^n$ or $2^n + 2^{n-1}$. 
When $j = 2^{n-1}$ or $2^n + 2^{n-1}$, properties 3 and 4 ensure that $\ww_j$ and $\ww_{j+1}$ are obtained from two consecutive $\vv_j$'s that have distance $n$.  When $j = 2^n$,  $\ww_j$ and $\ww_{j+1}$ are obtained from $\vv_{2^n}$ and $\vv_1$ (resp.),  which have distance $n$ due to properties 2 and 5. Hence, property 1 is satisfied for $\Pi_{n+1}$. 

It is easy to verify that properties 2--5 hold for $\Pi_{n+1}$, so the proof is complete.
\end{proof} 

We note that there is an alternative approach that gives an easier proof of Theorem \ref{ncube.thm} in the case where $n$ is even.
It is well-known
that, for $n$ even, the distance $n-1$ graph of $\QQ_n$ is isomorphic to  $\QQ_n$ (see, e.g., \cite[p.\ 265, Remark (i)]{BCN}). One such isomorphism is obtained by flipping the bits of 
all odd-weight $n$-tuples.  It is also well-known that  $\QQ_n$ has a hamiltonian cycle; the \emph{binary reflected gray code} is one example (see, e.g., \cite{graycode}). Hence it immediately follows that $\DL^\circ(Q_n)=\DL(\QQ_n)= n-1$ if $n$ is even. Note that this observation also yields the values of $\DL^\circ(Q_n)$ for even values of $n$. 

The distance $n-1$ graph of 
$\QQ_n$ is disconnected if $n$ is odd, so this particular approach will not work.

\subsection{Complete Binary Trees}
\label{cbt.sec}

The radius of the complete binary tree of depth $n$ is $n$. We will show that the complete binary tree of depth $n$ admits an $n$-dispersed labelling.

Let $T$ have vertex set $\{\varepsilon\} \cup \{0,1\} \cup \cdots \cup \{0,1\}^n$ where $\varepsilon$ denotes the empty string. We have an edge joining $\bx$ to $\by$ if 
one can be obtained from the other by deleting the rightmost bit. Let $|\bx|$ denote the length of the binary string $\bx$. 
A 3-dispersed labelling of a complete binary tree of depth three is given in Figure \ref{Fig:bintree3}.

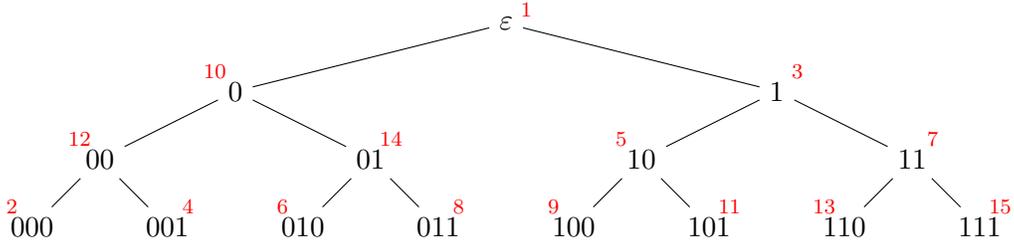
\begin{figure}[tb]
\begin{center}  
\begin{tikzpicture}[baseline=(000),node distance=0.5cm,xscale=0.9,yscale=0.9,    %
 solidvert/.style={draw,fill=red, circle, inner sep=1.5pt},
 hollowvert/.style={draw,circle, fill=white, inner sep=1.5pt},
 every loop/.style={min distance=10pt,in=-135,out=-45,looseness=10}]
 \def \n {3}
 \node (eps) at (0,\n) {$\varepsilon$}; 
\node (0) at (-4,\n-1) {$0$}; 
\node (1) at (4,\n-1) {$1$}; 
\node (00) at (-6,\n-2) {$00$}; 
\node (01) at (-2,\n-2) {$01$}; 
\node (10) at ( 2,\n-2) {$10$}; 
\node (11) at ( 6,\n-2) {$11$}; 
\node (000) at (-7,\n-3) {$000$}; 
\node (001) at (-5,\n-3) {$001$}; 
\node (010) at (-3,\n-3) {$010$}; 
\node (011) at ( -1,\n-3) {$011$}; 
\node (100) at ( 1,\n-3) {$100$}; 
\node (101) at ( 3,\n-3) {$101$}; 
\node (110) at ( 5,\n-3) {$110$}; 
\node (111) at (  7,\n-3) {$111$}; 
  \draw (eps) -- (0);   \draw (eps) -- (1); 
 \draw (0) -- (00);   \draw (0) -- (01);   \draw (1) -- (10);   \draw (1) -- (11); 
  \draw (00) -- (000);   \draw (00) -- (001);   \draw (01) -- (010);   \draw (01) -- (011);    \draw (10) -- (100);   \draw (10) -- (101);   \draw (11) -- (110);   \draw (11) -- (111); 
 \node[red] at (0.3,\n+0.2) {$\scriptstyle{1}$};
\node[red] at (-4.3,\n-1+0.3) {$\scriptstyle{10}$};
\node[red] at (4.3,\n-1+0.3) {$\scriptstyle{3}$};
\node[red] at (-6.3,\n-2+0.3) {$\scriptstyle{12}$};
\node[red] at (-1.7,\n-2+0.3) {$\scriptstyle{14}$};
\node[red] at (1.7,\n-2+0.3) {$\scriptstyle{5}$};
\node[red] at (6.3,\n-2+0.3) {$\scriptstyle{7}$};
\node[red] at (-7.3,\n-3+0.3) {$\scriptstyle{2}$};
\node[red] at (-4.7,\n-3+0.3) {$\scriptstyle{4}$};
\node[red] at (-3.3,\n-3+0.3) {$\scriptstyle{6}$};
\node[red] at (-0.7,\n-3+0.3) {$\scriptstyle{8}$};
\node[red] at (0.7,\n-3+0.3) {$\scriptstyle{9}$};
\node[red] at (3.3,\n-3+0.3) {$\scriptstyle{11}$};
\node[red] at (4.7,\n-3+0.3) {$\scriptstyle{13}$};
\node[red] at (7.3,\n-3+0.3) {$\scriptstyle{15}$};
 \end{tikzpicture}
 \caption{A $3$-dispersed labelling of the complete binary tree of depth three.
\label{Fig:bintree3}}
 \end{center}
 \end{figure}
 
 Our labelling is as follows. The root node, $\varepsilon$, receives label $1$. The ``left'' nodes---those vertices whose strings begin with zero---receive
 even labels, with the rule that labels $2,4,6,\ldots, 2^{n}$ are used on leaf nodes and the remaining even labels occur on nodes of degree three. The ``right'' nodes---those vertices whose strings begin with one---receive
 odd labels, with the rule that labels $2^n+1,2^n+3,\ldots, 2^{n+1}-1$ are used for leaf nodes and the remaining odd labels $\ell \ge 3$ occur on nodes of degree three.  
 
 A concrete rule for such a labelling is given as follows, where $N(\bx)$ denotes the integer whose binary representation is $\bx$. 
 We set 
 $$ \ell(\bx) = \begin{cases} 2N(\bx)  + 1& \mbox{if} \ x_1=1; \\
  2N(\bx) + 2 & \mbox{if} \ x_1=0 \ \mbox{and} \ \bx \ \mbox{is a leaf node;} \\
  2N(\bx) +  2^n+2^{|\bx|} & \mbox{if} \ x_1=0 \ \mbox{and} \ \bx \ \mbox{is not a leaf node.}
 \end{cases}$$
 One may verify by inspection that each label $0<i<2^{n+1}$ is used once. Among any consecutive pair of labels $i,i+1$, one of these appears
 at a leaf node and the path from this node to the other one passes through the root of the tree and therefore this path has length at least $n$.
 
 Therefore we have proven the following result.
 
 \begin{Theorem}
 \label{cbt.thm}
 Let $T_n$ denote the complete binary tree of depth $n$. Then $\DL(T_n)= n$.
 \end{Theorem}
 
 \begin{proof}
 The labelling described above shows that $\DL(T_n) \geq n$. Since $T_n$ has radius $n$, it follows immediately that $\DL(T_n)= n$.
 \end{proof}
 
\section{A product construction}
\label{sec3}

The \emph{Cartesian product} of  graphs $G = (V(G), E(G))$ and $H = (V(H),E(H))$, denoted $G\,\Box \,H$, is the graph with 
vertex set $V(G)\times V(H)$ and edge set 
$$E(G \, \Box \, H) = \left\{ \{ (u,v), (u',v') \} : \  u=u', \ \{v,v'\} \in E(H)  \ \mbox{or} \ \{u,u'\}\in E(G),  v =v' \ \right\}. $$

Suppose $G$ is a graph on $m$ vertices having a $k$-circular-dispersed labelling and 
$H$ is a graph on $n$ vertices having a  $k'$-circular-dispersed labelling. When $\gcd \left( m,n \right) =1$, we will show how  to combine these two labellings to give a $(k+k')$-circular-dispersed labelling of 
the product graph $G \,\Box \,H$.



\begin{Theorem}
\label{product.thm}
Let $G$ and $H$ be graphs with  $|V(G)|=m$, $|V(H)| =n$ and $\gcd \left( m,n \right) =1$. Then $\DL^\circ( G\, \Box \, H) \ge \DL^\circ(G)+\DL^\circ(H)$.
\end{Theorem}

\begin{proof}
Denote $k = \DL^\circ(G)$ and $k' = \DL^\circ(H)$.  Suppose $\ell_G:    \zed_m \rightarrow V(G)$ is a $k$-circular-dispersed labelling of $G$ and $\ell_H:   \zed_n  \rightarrow V(H)$ is a $k'$-circular-dispersed labelling of $H$. 
For $0 \leq i \leq mn-1$, define
\[  \ell(i) = ( \ell_G(i \bmod m),\ell_H(i \bmod n) ).\]
  Since $\gcd \left( m,n \right) =1$, it follows immediately that 
  $\ell$ is a bijection from $\zed_{mn} $ to the vertices of  $G\,\Box \,H$.
The distance between two vertices in $G\,\Box \,H$ is just the sum of the distances between the corresponding projections (i.e., to the first and second coordinates) in $G$ and $H$.
Hence, for any $i$, where $0 \leq i \leq mn-1$, The distance between $\ell(i)$ and $\ell(i+1 \bmod mn)$ in $G\,\Box \,H$ is the sum 
\[ d_G(i \bmod m,(i+1) \bmod m) + d_H(i \bmod n,(i+1) \bmod n) \geq  k + k',\] as desired.
\end{proof}

We illustrate the product construction by returning to grids. Our goal is not to compute all the values $\mathsf{DL}^{\circ}(L_{m,n})$, but rather to illustrate how Theorem \ref{product.thm} can be applied to a specific family of graphs.

Recall that $L_{m,n}$ is an $m$ by $n$ grid and $P_{m-1}$ is a path having $m$ vertices.
It is easy to see that $P_{m-1} \, \Box \, P_{n-1}$ is isomorphic to $L_{m,n}$. We can obtain lower bounds on $\mathsf{DL}^{\circ}(L_{m,n})$
if we have lower bounds on $\mathsf{DL}^{\circ}(P_{m-1})$ and $\mathsf{DL}^{\circ}(P_{n-1})$. Therefore, we first look at circular-dispersed labellings of paths. 

%

\begin{Lemma}
\label{two.lem}
If $\mathsf{DL}^{\circ}(G) = r(G)$, then $G$ has no uniquely eccentric central vertices.
\end{Lemma}

\begin{proof}
Let $n=|V|$ and suppose $\ell : \zed_n \rightarrow V$  is a $k$-circular-dispersed labelling with $k$ as large as possible. Let $x$ be a uniquely eccentric central vertex and suppose $\ell(i)=x$. Then at least one of $d( \ell(i-1),x)$, $d(x,\ell(i+1))$ is smaller than $r(G)$. So $k< r(G)$.
\end{proof}

\begin{Lemma}
\label{circpath.lem}
Let $P_m$ be the path with $m$ edges (and $m+1$ vertices). If $m$ is even, then $\DL^\circ(P_m)=\DL(P_m)= \frac{m}{2}$. If $m \geq 3$ is odd, then
 $\DL^\circ(P_m)=\DL(P_m)-1= \frac{m-1}{2}$. 
\end{Lemma}

\begin{proof}
We already have shown that  $\DL(P_m) = r(P_m)$ in Theorem \ref{paths.thm}. For $m$ even, the labelling 
\[ 2 \:\:\: 4 \:\:\: \cdots \:\:\: m \:\:\: 1 \:\:\: 3 \:\:\: \cdots \:\:\: m+1\]
given in Theorem \ref{paths.thm} 
has $d(1,m+1)=\frac{m}{2}$. Hence, $\DL^\circ(P_m)= \DL(P_m) = \frac{m}{2}$ when $m$ is even. 


However, when $m \geq 3$ is odd, $P_m$ contains a uniquely eccentric central vertex (actually, it contains two such vertices).
Hence, Lemma \ref{two.lem} asserts that
$\DL^\circ(P_m) \le r(P_m) - 1 = (m-1)/2$.
To achieve a $\frac{m-1}{2}$-circular dispersed labelling of $P_m$, we start with the above labelling of $P_{m-1}$ (on $m$ vertices) and attach a new vertex labelled 
 $m+1$ to the left end of the path.
\end{proof}

\begin{Theorem}
\label{gridprod.thm}
Suppose $\gcd(m,n) = 1$, $m \geq 3$ and $n \geq 3$. Then 
\[ 
\mathsf{DL}^{\circ}(L_{m,n}) \geq
\begin{cases}
\frac{m+n}{2} - 1 & \text{if $m,n$ are both odd}\\
\frac{m+n-3}{2} & \text{if $m+n$ is odd}.
\end{cases}
\]
\end{Theorem}

\begin{proof}
This is an immediate application of Theorem \ref{product.thm} and Lemma \ref{circpath.lem}.
\end{proof}

We do not consider the case where both $m$ and $n$ are even in Theorem \ref{gridprod.thm}. This is because
$\gcd(m,n) \geq 2$ when $m$ and $n$ are even and hence the hypotheses cannot be satisfied. 

\begin{Corollary}
Suppose $\gcd(m,n) = 1$, $m \geq 3$ is odd and $n \geq 3$ is odd. Then 
\[ 
\mathsf{DL}^{\circ}(L_{m,n}) =
\frac{m+n}{2} - 1.\]
\end{Corollary}

\begin{proof}
From Theorem \ref{oddm.thm}, we have $\mathsf{DL}(L_{m,n}) =
\frac{m+n}{2} - 1$ when $n\geq 3$ and $m\geq 3$ are odd. 
Lemma \ref{circ.lem} shows that $\mathsf{DL}^{\circ}(L_{m,n}) \leq \mathsf{DL}(L_{m,n})$. Finally, 
Theorem \ref{gridprod.thm} shows that $\mathsf{DL}^{\circ}(L_{m,n}) \geq
\frac{m+n}{2} - 1 $ if $n\geq 3$ and $m\geq 3$ are odd and $\gcd(m,n) = 1$. 
\end{proof}

\section{Degree-based bounds}
\label{sec4}


We now consider a degree-based lower bound on $\DL(G)$. Suppose $G$ has $n$ vertices.
For $x\in V$  and $i\ge 0$, let 
$$\eta_i(x) = \left| \left\{ y\in V : \ d(x,y) = i\right\} \right| ,$$and
$$ \kappa(x) = \max \left\{ j : \sum_{i=0}^j \eta_i(x) \le n/2 \right\}.$$
For example, in $\QQ_m$, it is easy to see that $\eta_i(x)= \binom{m}{i}$  and $\kappa(x) = \lfloor \frac{m-1}{2} \rfloor$ for each vertex $x$.

\begin{Lemma}
\label{Lem5}
$\DL(G) \ge \min \{ 1+ \kappa(x) : x \in V(G) \}$.
\end{Lemma}

\begin{proof}
For any vertex $x$,  there are at least $n/2$ vertices $y$  such that $d(x,y) \geq 1+ \kappa(x)$.  
Therefore, by Dirac's Theorem (e.g., see \cite[Theorem 18.4]{BM}), $H_k(G)$ is hamiltonian for $k = \min \{ 1+ \kappa(x) :x \in V(G) \}$. Now apply Theorem \ref{Thm1}.
\end{proof}

\begin{Theorem}
\label{Thm6}
If $G$ has $n$ vertices and maximum degree $\Delta > 2$, then 
\begin{equation}
\label{eq1} \DL(G) \ge 1 + \log_{\Delta-1} \left[ 1 + \frac{ (n-2)(\Delta-2) }{2\Delta} \right].
\end{equation}
\end{Theorem}

\begin{proof}
For any $x\in V(G)$ and $i\ge 1$, $\eta_i(x) \le \Delta (\Delta-1)^{i-1}$. In order to apply 
Lemma \ref{Lem5}, we compute the largest value of $\ell$ such that  
$$ 1 + \Delta + \Delta(\Delta-1) + \cdots + \Delta (\Delta-1)^{\ell} \le n/2. $$
This inequality is equivalent to  
\[1+ \frac{ \Delta (1- (\Delta-1)^{\ell+1} )}{ 2-\Delta} \le \frac{n}{2}.\]
Elementary algebra tells us that this can be rewritten as 
$$\ell+1 \ge   \log_{\Delta-1} \left[ 1 + \frac{ (n-2)(\Delta-2) }{2\Delta} \right].$$
Since $\kappa(x) \geq\ell+1$, it
therefore follows from Lemma \ref{Lem5} that $\DL(G) \geq \ell + 2$ and hence the inequality (\ref{eq1}) holds.
\end{proof}


\begin{Corollary}
\label{cor4.3} If $G$ has $n$ vertices and maximum degree three, then  \[\DL(G) \geq \log_2 \left( \frac{n+4}{6} \right) +1.\]
\end{Corollary}

\begin{proof}
Apply Theorem \ref{Thm6}  with $\Delta = 3$. 
\end{proof}

To illustrate the application of Corollary
\ref{cor4.3}, we consider complete binary trees, which we studied in Section \ref{cbt.sec}. 
A complete binary tree of depth $m$, denoted $T_{m}$, has $2^{m+1}-1$ vertices and its maximum degree is three.
Hence, Corollary \ref{cor4.3} asserts that 
\[\DL(T_{m}) \geq \log_2 \left( \frac{2^{m+1}+3}{6} \right) +1 > m - 0.6.\]
Since $\DL(T_{m})$ is an integer, it follows that $\DL(T_{m}) \geq m$.
The exact value of $\DL(T_{m})$, as given by Theorem \ref{cbt.thm}, is $m$. Hence, in the special case of complete binary trees, Corollary
\ref{cor4.3} is tight.



\section{Further questions}
\label{sec5}

There are many interesting  questions regarding dispersed labellings. We list some now.

\begin{enumerate}
\item Which graphs $G$ have $\mathsf{DL}(G) = r(G)$?
\item We note that $\mathsf{DL}(G) \in \{r(G), r(G)-1\}$ for all the classes of graphs considered in Section \ref{sec2}.
It is interesting to ask if there is an infinite family of graphs $G_1, G_2, \dots$ such that
\[\lim _{i \rightarrow 
\infty} (r(G_i) - \mathsf{DL}(G_i)) = \infty.\]
John Haslegrave (private communication) has answered this question in the affirmative, as follows.
Suppose we attach two paths of length $k-1$ to distinct vertices of a clique of order $2k$. This graph has radius $k$, but for any ordering of vertices some two consecutive vertices are in the clique. Hence, there is no $2$-dispersed labelling. 

We note that Deepak Bal and Sarah Acquaviva also have found a construction that provides a positive answer to this question.
\item Is there an efficient algorithm to find ``good'' dispersed labellings of graphs (i.e., labellings  
where  $\mathsf{DL}(G)$ is ``close to'' $r(G)$) for certain classes of graphs (e.g., trees)? 
\item Which (rooted) binary trees $T$ of depth $m$ have $\mathsf{DL}(T) < m$?
\item Is there a generalization of the product construction (Theorem \ref{product.thm}) that handles cases where $\gcd(m,n) > 1$?
\item Is there a product-type construction that yields the correct values for $\DL(\QQ_n)$?
\item There exist  sufficient conditions for a graph to be hamiltonian other than Dirac's Theorem. Do any of these lead to interesting lower bounds on $\mathsf{DL}(G)$ for certain graphs $G$?
\end{enumerate}

\section*{Acknowledgements}

We thank the referees for detailed and valuable comments. In particular, we are grateful to one referee for pointing out an error in Theorem 2.10 in an earlier version of the paper. We also thank John Haslegrave, Deepak Bal and Sarah Acquaviva for communicating their affirmative answers to question 2 in Section \ref{sec5}.

\end{document}